\documentclass[a4paper,12pt]{amsart}

\usepackage{unicode-math}

\usepackage{graphics,graphicx}
\usepackage{tikz-cd}
\usetikzlibrary{calc}

\usepackage{xcolor}

\usepackage{enumitem}

\usepackage{booktabs}

\usepackage[a4paper]{geometry}
\usepackage[backref=page,colorlinks=true,allcolors=blue]{hyperref} 
\newtheorem{thm}{Theorem}[section]

\theoremstyle{definition}
\newtheorem{defn}[thm]{Definition}
\newtheorem{rem}[thm]{Remark}

\newtheorem{exa}[thm]{Example}

\DeclareMathOperator{\GL}{GL}

\DeclareMathOperator{\Fan}{Fan}

\newcommand{\bbZ}{\mathbb{Z}}
\newcommand{\bbP}{\mathbb{P}}
\newcommand{\bbR}{\mathbb{R}}
\newcommand{\bbN}{\mathbb{N}}

\newcommand{\bbF}{\mathbb{F}}

\begin{document}

\title{The Elusive Relatively K-Unstable Delzant Octagon: \\ A Numerical Search}
\author{Thibaut Delcroix}
\address{Univ Montpellier, CNRS, Montpellier, France}
\email{thibaut.delcroix@umontpellier.fr}
\urladdr{https://delcroix.perso.math.cnrs.fr/}
\author{Bijan Mohammadi}
\address{Univ Montpellier, CNRS, Montpellier, France}
\email{bijan.mohammadi@umontpellier.fr}
\urladdr{https://imag.umontpellier.fr/~mohamadi/}
\date{2026}


\begin{abstract}
Relative K-stability of toric varieties can be tested through explicit Donaldson-Futaki computations, yet finding unstable examples remains challenging. We develop a reproducible symbolic-numeric framework to investigate all fifteen families of Delzant octagons corresponding to smooth toric surfaces of Picard rank six. While numerical optimization identifies several apparently destabilizing configurations, exact rational reconstruction systematically removes the observed instability.
We explain why Wang and Zhou's 2014 tentative construction of a relatively K-unstable octagon does not allow to find an explicit example. 
\end{abstract}


\maketitle

\section{Introduction}

This work originated from a series of discussions between a pure mathematician and a physicist with a background in numerical methods, combining geometric analysis, computational experimentation, and physical intuition and interpretation. 
Beyond the specific problem of relative K-stability, this work serves as a case study illustrating the interaction between numerical optimization, symbolic certification, and experimental discovery in algebraic geometry.

A central problem in Kähler geometry is to determine which complex manifolds
admit canonical Kähler metrics such as Kähler-Einstein metrics, constant scalar curvature (cscK) metrics, and more generally 
extremal metrics introduced by Calabi \cite{Calabi1982}.
For Kähler manifolds, the Yau-Tian-Donaldson (YTD) correspondence very recently established in \cite{Li_2022_cscK,Boucksom-Jonsson,Darvas-Zhang,Trusiani} relates the
existence of such canonical metrics to a stability notion called uniform relative \emph{K-stability} (or variants).
The original precise formulation of the Yau-Tian-Donaldson correspondence by Donaldson in \cite{Donaldson_2002} does not hold in general, as shown by Jihao Liu using LLMs \cite{Liu}. 

In the case of toric manifolds, Donaldson reformulated K-stability as a convex-geometric problem on the associated Delzant polytope
\cite{Donaldson_2002}. 
This approach translates the stability question into explicit integral inequalities of piecewise linear convex functions over the polytope, making toric surfaces suitable test cases for the YTD correspondence. 
In subsequent works culminating in \cite{Don09}, he established that K-stability (as opposed to uniform K-stability) is indeed sufficient for the existence of cscK metrics on toric surfaces, and furthermore that it is enough to test on piecewise linear function with two linearity domains.  
This is one of the very few classes of polarized varieties (with indefinite first Chern class) were relative K-stability may be approached numerically in a satisfying way: the space of conditions to check is finite dimensional. 
Following the first author's terminology, one may call this an example of an effective YTD correspondence. 

More generally, the case of relative K-stability for toric surfaces, corresponding to existence of extremal Kähler metrics, has been addressed by several authors, among which \cite{Wang-Zhou_2011, Wang-Zhou_2014,ChenLiSheng2014,Li_Lian_Sheng_2023}. 
Other examples of effective YTD correspondences include \cite{ACGT_2008, Delcroix_2023_RK1, Jubert_2023, Jubert-Yin}. 

However satisfying this result may be, it is in general very hard to check relative K-stability even for toric surfaces, in contrast with the case of Fano varieties where, at least for very symmetric varieties, checking K-stability reduces simple barycenter conditions \cite{Wang-Zhu_2004,Delcroix_2017_KE,Ilten-Suss_2017,Delcroix_2020,Li_Li_2026}. 
In his original article \cite{Donaldson_2002}, Donaldson exhibits a first example of relatively K-unstable toric surface. 
The articles \cite{Wang-Zhou_2011, Wang-Zhou_2014} provide further examples, while Luo and Chen
\cite{LuoChen2020} devote a full article to verifying relative K-stability for all Kähler classes on the blowup of \(\bbP^2\) at two points. 
It is in principle easy to provide a certificate that a toric surface is relatively K-unstable: it suffices to find a destabilizing piecewise linear function. 
However, it is not so easy to find toric surfaces which may be relatively K-unstable. 
In \cite{Wang-Zhou_2011} the authors exhibit, by an asymptotic argument, relative K-unstable toric surfaces with Picard rank 7, and claim that there probably does not exist any examples with Picard rank 5. 
In \cite{Wang-Zhou_2014} they claim to have found examples of relatively K-unstable toric surfaces with Picard rank 6. 

The present work investigates the exhaustive collection of fifteen families of Delzant
octagons, each depending on an integer parameter and six positive variables (five up to rescaling). 
By toric geometry, these correspond to all Kähler classes on smooth, Picard rank 6 toric surfaces.  
The objective is to develop and test a computational framework capable of identifying candidate destabilizing configurations and certifying them through exact verification. 
This amounts to obtaining a negative value for the Donaldson-Futaki invariant. 

To this end, a numerical optimization procedure is developed to explore the
parameter space and identify candidate destabilizing configurations. The
implementation allows computations to be carried out at user-prescribed
floating-point precision and permits a systematic exploration of the parameter.
For each candidate configuration, the associated affine function is
computed, a convex piecewise-linear test function is constructed, and
the Donaldson-Futaki invariant is evaluated.

Among the fifteen families considered, this procedure initially provides several configurations which exhibit negative values of Donaldson-Futaki invariant. 
However, it turns out to be a numerical artifact: the numerical instabilities do not persist after exact rational reconstruction and symbolic verification. 
In particular, we were not able to replicate Wang and Zhou's octagon examples. 

Conversely, as should be expected from a numerical investigation, our algorithm does not provide a proof that all octagons considered are relatively K-stable, only a strong indication that they should be. 
It should be noted that, in his PhD thesis \cite{Sektnan_2016_PhD}, Lars Sektnan developed an algorithm that, in some cases (namely in situations where there are no vector field obstructions to the existence of cscK metrics), is theoretically conclusive. 
In the general relative K-stability case, though, no such algorithm exists. 

The paper concludes with a discussion of the links between the toric framework, Green's formulas, and Poisson theory, together with a series of heuristic physical interpretations of the main results. These include viewing the extremal affine function as an effective interior density induced by boundary data, and interpreting the negativity of the Donaldson functional as the emergence of an energy-lowering instability mode.

The computational framework is publicly available. 

\section{Relative K-unstability of toric surfaces}

Let \(\Delta\subset \bbR^2\) be a convex polygon with rational normal fan, that is, such that one may write 
\[ \Delta = \bigcap_{0\leq i \leq \rho+2} \{\langle u_i, \cdot\rangle \geq c_i \} \]
where the \(u_i\) are primitive vectors in \(\bbZ^2\), \(c_i\in \bbR\), \(\rho\in \bbN\) such that \(\rho+2\) is exactly the number of facets of \(\Delta\). 
We denote by \(F_i\) the facet \(\Delta\cap \{\langle u_i, \cdot\rangle = c_i \}\). 

By the general theory of toric varieties \cite{Fulton_1993}, \(\Delta\) encodes a toric surface \(X\) equipped with a Kähler class \(\alpha\). 
If furthermore the numbers \(c_i\) are integers, then \(\alpha=c_1(L)\) for an ample line bundle \(L\) on \(X\). 
The toric surface is smooth if the polytope is Delzant, that is, for any two \emph{adjacent} facets \(F_i\) and \(F_j\), the corresponding normals \(u_i\), \(u_j\) form a basis of \(\bbZ^2\). 
In this case, the number \(\rho\) is the Picard rank of the toric surface. 
We will classify the Delzant octagons (corresponding to Kähler classes on smooth toric surfaces with Picard rank 6) in the next section. 

To consider K-stability of \((X,L)\), we must introduce a few notions. 
Let \(\partial \Delta\) denote the boundary of \(\Delta\), which is the union of its facets. 
We consider on \(\partial \Delta\) the measure \(d\sigma\), defined as follows. 
The affine span of a facet \(F_i\) is an affine line with a rational direction \(vec{F_i}\), so \(\bbZ^2\cap \vec{F_i}=\bbZ v_i\) for some \(v_i\in \bbZ^2\). 
On \(\vec{F_i}\) there is a natural normalized Lebesgue measure which gives mass \(1\) to the segment \([0,v_i]\), and \(d\sigma|_{F_i}\) is the pullback of this measure by translation. 

Consider the linear functional 
\( \mathcal{L}: C^0(\Delta,\bbR) \to \bbR \) defined by 
\[ \mathcal{L}(f) = \int_{\partial \Delta} f d\sigma - \int_{\Delta} (a_1x_1+a_2x_2+a_0)f(x_1,x_2) dx_1dx_2 \]
where \(a_0\), \(a_1\) and \(a_2\) are the unique real numbers such that 
\[ \mathcal{L}(p_1) = \mathcal{L}(p_2) = \mathcal{L}(\mathbb{1}) = 0 \]
where \(p_i:(x_1,x_2)\mapsto x_i\) are the projections and \(\mathbb{1}\) is the constant function equal to \(1\). 

\begin{defn}
The polygon \(\Delta\) is \emph{relatively K-unstable} if there exists a convex function \(f:\Delta\to \bbR\) such that \( \mathcal{L}(f)<0\). 
\end{defn}

\begin{rem}
It is not hard to check (but this is somewhat hidden in the definition of the numbers \(a_i\)), that for any positive number \(t\), \(t\Delta\) is relatively K-unstable if and only if \(\Delta\) is. 
\end{rem}

We are interested in detecting relatively K-unstable Delzant octagons. 
The space of all convex functions on \(\Delta\) is infinite dimensional, we will focus on a smaller class of convex functions. 

\begin{defn}
A \emph{simple piecewise linear convex function} on \(\Delta\) is a function of the form 
\(f(x_1,x_2) = \sup(0,b_1x_1+b_2x_2+b_0) \)
for some real numbers \(b_1\), \(b_2\) and \(b_0\) such that the affine line \(\{(x_1,x_2)\mid b_1x_1+b_2x_2+b_0=0\}\) intersects the interior of \(\Delta\). 
The affine line \(\{(x_1,x_2)\mid b_1x_1+b_2x_2+b_0=0\}\)  is called the \emph{crease} of \(f\). 
\end{defn}

Since we are interested in numerical experiments, we are entitled to focus on simple piecewise linear functions which are simple examples of convex functions, to detect relative K-unstability. 
Nonetheless, a remarkable result of Donaldson \cite{Donaldson_2002}, Wang and Zhou \cite{Wang-Zhou_2011,Wang-Zhou_2014} states that it is enough to consider simple piecewise linear functions to detect relative K-unstability, and also K-stability.

\section{Delzant octagons}

Recall \cite{Fulton_1993} that complete toric varieties of complex dimension \(n\) are encoded by rational polyhedral fans in \(\bbR^n\) with full support, up to the action of \(\GL_n(\bbZ)\). 
In dimension \(2\), a complete fan is fully described by the data of the (primitive) generators of all rays, we denote by \(\Fan(u_1,\ldots,u_r)\) the complete fan such that its rays are generated by \(u_1, \ldots, u_r\). 
For example, the fan of \(\bbP^2\) is given by the fan whose rays are generated by \((1,0)\), \((0,1)\) and \((-1,-1)\), and the fan of the Hirzebruch surface \(\bbF_k=\bbP_{\bbP^1}(\mathcal{O}\oplus\mathcal{O}(k))\) is given by the fan whose rays are generated by \((1,0)\), \((0,1)\), \((0,-1)\) and \((-1,k)\).
Using the notations \(e_1=(1,0)\), \(e_2=(0,1)\), and \(v_k=(-1,k) =-e_1+ke_2\) throughout, the fan of \(\bbP^2\) is thus \(\Fan(e_1,e_2,-e_1-e_2)\) and the fan of \(\bbF_k\) is \(\Fan(\pm e_2, e_1, v_k)\).

\begin{thm}
A smooth projective toric surface of Picard rank 6 corresponds to one of the fans \(\Sigma_i(k)\) for some \(k\in \bbZ\) and \(i\in \{1,\ldots,15\}\) below. 
A Kähler class on such a surface corresponds to one of the Delzant octagons \(\Delta_i(k,\ell_1,\ldots, \ell_6)\) given below, which satisfies one or two condition listed with the description of the polytope. 
Conversely, to each choice of \(\Delta_i(k,\ell_1,\ldots, \ell_6)\) with admissible parameters corresponds a Kähler class on the smooth projective toric surface of Picard rank \(6\) with fan \(\Sigma_i(k)\). 
\end{thm}

\begin{proof}

It is well known (see \cite[page 43, Proposition]{Fulton_1993}) that a smooth projective toric surface can be obtained by a succession of blowups at torus fixed points from the projective plane \(\bbP^2\) or a Hirzebruch surface \(\bbF_k\) for some \(k\in \bbZ_{\geq 0}\). 
Torus fixed points correspond to maximal cones in the fan, and in our two dimensional setting, blowing up the torus fixed point corresponding to the cone generated by primitive generators \(u_1\) and \(u_2\) amounts to adding the primitive ray generator \(u_1+u_2\). 
In particular, the blowup of \(\bbP^2\) is (up to symmetries) the toric surface whose fan is \(\Fan(e_1,e_2,-e_2=-e_1-e_2+e_1,-e_1-e_2)\), which coincides with the fan of \(\bbF_1\) (up to symmetries again). 
As a consequence, it suffices to consider blowups of Hirzebruch surfaces as soon as the Picard rank of the smooth projective toric surface is \(2\) or higher. 

At the level of fans, it implies in particular that there exists a pair of opposite rays (which is apart from sporadic exceptions, unique) which may be assumed to be generated by \(\pm e_2\). 
Considering the various possibilities of blowups, one quickly reaches a classification of all possible fans for a given Picard rank.  
Smooth Picard rank \(6\) toric surfaces correspond to one of fifteen families of fans \(\Sigma_i(k)\) parametrized by an integer \(k\) as follows. 

\begin{itemize}
\item \(\Sigma_1(k) = \Fan(\pm e_2,e_1,v_k,v_k+e_2,v_k+2e_2,v_k+3e_2,v_k+4e_2)\)
\item \(\Sigma_2(k) = \Fan(\pm e_2,e_1,v_k,v_k+e_2,v_k+2e_2,v_k+3e_2,2v_k+e_2)\)
\item 
\(\Sigma_3(k) = \Fan(\pm e_2,e_1,v_k,v_k+e_2,v_k+2e_2,v_k+3e_2,2v_k+3e_2)\)
\item 
\(\Sigma_4(k) = \Fan(\pm e_2,e_1,v_k,v_k+e_2,v_k+2e_2,2v_k+e_2,3v_k+2e_2)\)
\item 
\(\Sigma_5(k) = \Fan(\pm e_2,e_1,v_k,v_k+e_2,v_k+2e_2,2v_k+e_2,3v_k+e_2)\)
\item 
\(\Sigma_6(k) = \Fan(\pm e_2,e_1,v_k,v_k+e_2,v_k+2e_2,2v_k+e_2,2v_k+3e_2)\)
\item 
\(\Sigma_7(k) = \Fan(\pm e_2,e_1,v_k,v_k+e_2,2v_k+e_2,3v_k+e_2,4v_k+e_2)\)
\item 
\(\Sigma_8(k) = \Fan(\pm e_2,e_1,v_k,v_k+e_2,2v_k+e_2,3v_k+e_2,5v_k+2e_2)\)
\item 
\(\Sigma_9(k) = \Fan(\pm e_2,e_1,v_k,v_k+e_2,2v_k+e_2,3v_k+e_2,3v_k+2e_2)\)
\item 
\(\Sigma_{10}(k) = \Fan(\pm e_2,e_1,e_1-e_2,v_k,v_k+e_2,v_k+2e_2,v_k+3e_2)\)
\item 
\(\Sigma_{11}(k) = \Fan(\pm e_2,e_1,e_1-e_2,v_k,v_k+e_2,v_k+2e_2,2v_k+e_2)\)
\item 
\(\Sigma_{12}(k) = \Fan(\pm e_2,e_1,e_1-e_2,v_k,v_k+e_2,2v_k+e_2,3v_k+e_2)\)
\item 
\(\Sigma_{13}(k) = \Fan(\pm e_2,e_1,e_1-e_2,e_1-2e_2,v_k,v_k+e_2,v_k+2e_2)\)
\item 
\(\Sigma_{14}(k) = \Fan(\pm e_2,e_1,e_1-e_2,e_1-2e_2,v_k,v_k+e_2,2v_k+e_2)\)
\item 
\(\Sigma_{15}(k) = \Fan(\pm e_2,e_1,e_1-e_2,2e_1-e_2,v_k,v_k+e_2,2v_k+e_2)\)
\end{itemize}

For the toric surface whose fan is \(\Sigma_i(k)\), Kähler classes correspond to (convex) octagons whose normal inner fan is \(\Sigma_i(k)\), up to translation in the plane (sometimes one takes the opposite covnention: the normal outer fan). Those can be parametrized by the length of six edges (with some conditions on the admissible lengths) and the lengths of the two remaining edges are then prescribed. 
To deal consistently with the various cases, we will assume that the origin \((0,0)\) is the common vertex of the two edges with prescribed lengths, and we denote by \(\ell_1,\ldots,\ell_6\) the lattice length of the six other edges, ordered counterclockwise. 
We denote the corresponding octagon by 
\(\Delta_i(k,\ell_1,\ldots,\ell_6)\).

The octagon \(\Delta_1(k,\ell_1,\ldots,\ell_6)\) corresponds to a Kähler class if \(\ell_6>(k+4)\ell_1+(k+3)\ell_2+(k+2)\ell_3+(k+1)\ell_4+k\ell_5\), and then it has the following eight vertices
\begin{itemize}
\item \((0,0)\), 
\item \((\ell_6-(k+4)\ell_1-(k+3)\ell_2-(k+2)\ell_3-(k+1)\ell_4-k\ell_5,0)\), 
\item \((\ell_6-(k+3)\ell_2-(k+2)\ell_3-(k+1)\ell_4-k\ell_5,\ell_1)\), 
\item \((\ell_6-(k+2)\ell_3-(k+1)\ell_4-k\ell_5,\ell_1+\ell_2)\), 
\item \((\ell_6-(k+1)\ell_4-k\ell_5,\ell_1+\ell_2+\ell_3)\), 
\item \((\ell_6-k\ell_5,\ell_1+\ell_2+\ell_3+\ell_4)\), 
\item \((\ell_6,\ell_1+\ell_2+\ell_3+\ell_4+\ell_5)\), 
\item \((0,\ell_1+\ell_2+\ell_3+\ell_4+\ell_5)\). 
\end{itemize}

The octagon \(\Delta_2(k,\ell_1,\ldots,\ell_6)\) corresponds to a Kähler class if \(\ell_6>(k+3)\ell_1+(k+2)\ell_2+(k+1)\ell_3+(2k+1)\ell_4+k\ell_5\), and then it has the following eight vertices
\begin{itemize}
\item \((0,0)\), 
\item \((\ell_6-(k+3)\ell_1-(k+2)\ell_2-(k+1)\ell_3-(2k+1)\ell_4-k\ell_5,0)\), 
\item \((\ell_6-(k+2)\ell_2-(k+1)\ell_3-(2k+1)\ell_4-k\ell_5,\ell_1)\), 
\item \((\ell_6-(k+1)\ell_3-(2k+1)\ell_4-k\ell_5,\ell_1+\ell_2)\), 
\item \((\ell_6-(2k+1)\ell_4-k\ell_5,\ell_1+\ell_2+\ell_3)\), 
\item \((\ell_6-k\ell_5,\ell_1+\ell_2+\ell_3+2\ell_4)\), 
\item \((\ell_6,\ell_1+\ell_2+\ell_3+2\ell_4+\ell_5)\), 
\item \((0,\ell_1+\ell_2+\ell_3+2\ell_4+\ell_5)\). 
\end{itemize}

The octagon \(\Delta_3(k,\ell_1,\ldots,\ell_6)\) corresponds to a Kähler class if \(\ell_6>(k+3)\ell_1+(k+2)\ell_2+(2k+3)\ell_3+(k+1)\ell_4+k\ell_5\), and then it has the following eight vertices
\begin{itemize}
\item \((0,0)\), 
\item \((\ell_6-(k+3)\ell_1-(k+2)\ell_2-(2k+3)\ell_3-(k+1)\ell_4-k\ell_5,0)\), 
\item \((\ell_6-(k+2)\ell_2-(2k+3)\ell_3-(k+1)\ell_4-k\ell_5,\ell_1)\), 
\item \((\ell_6-(2k+3)\ell_3-(k+1)\ell_4-k\ell_5,\ell_1+\ell_2)\), 
\item \((\ell_6-(k+1)\ell_4-k\ell_5,\ell_1+\ell_2+2\ell_3)\), 
\item \((\ell_6-k\ell_5,\ell_1+\ell_2+2\ell_3+\ell_4)\), 
\item \((\ell_6,\ell_1+\ell_2+2\ell_3+\ell_4+\ell_5)\), 
\item \((0,\ell_1+\ell_2+2\ell_3+\ell_4+\ell_5)\). 
\end{itemize}

The octagon \(\Delta_4(k,\ell_1,\ldots,\ell_6)\) corresponds to a Kähler class if and only if \(\ell_6-k\ell_5-(2k+1)\ell_4-(3k+2)\ell_3-(k+1)\ell_2-(k+2)\ell_1>0\), and then it has the following eight vertices
\begin{itemize}
\item \((0,0)\), 
\item \((\ell_6-k\ell_5-(2k+1)\ell_4-(3k+2)\ell_3-(k+1)\ell_2-(k+2)\ell_1,0)\), 
\item \((\ell_6-k\ell_5-(2k+1)\ell_4-(3k+2)\ell_3-(k+1)\ell_2,\ell_1)\), 
\item \((\ell_6-k\ell_5-(2k+1)\ell_4-(3k+2)\ell_3,\ell_1+\ell_2)\), 
\item \((\ell_6-k\ell_5-(2k+1)\ell_4,\ell_1+\ell_2+3\ell_3)\), 
\item \((\ell_6-k\ell_5,\ell_1+\ell_2+3\ell_3+2\ell_4)\), 
\item \((\ell_6,\ell_1+\ell_2+3\ell_3+2\ell_4+\ell_5)\), 
\item \((0,\ell_1+\ell_2+3\ell_3+2\ell_4+\ell_5)\). 
\end{itemize}

The octagon \(\Delta_5(k,\ell_1,\ldots,\ell_6)\) corresponds to a Kähler class if and only if \(\ell_6-k\ell_5-(3k+1)\ell_4-(2k+1)\ell_3-(k+1)\ell_2-(k+2)\ell_1>0\), and then it has the following eight vertices
\begin{itemize}
\item \((0,0)\), 
\item \((\ell_6-k\ell_5-(3k+1)\ell_4-(2k+1)\ell_3-(k+1)\ell_2-(k+2)\ell_1,0)\), 
\item \((\ell_6-k\ell_5-(3k+1)\ell_4-(2k+1)\ell_3-(k+1)\ell_2,\ell_1)\), 
\item \((\ell_6-k\ell_5-(3k+1)\ell_4-(2k+1)\ell_3,\ell_1+\ell_2)\), 
\item \((\ell_6-k\ell_5-(3k+1)\ell_4,\ell_1+\ell_2+2\ell_3)\), 
\item \((\ell_6-k\ell_5,\ell_1+\ell_2+2\ell_3+3\ell_4)\), 
\item \((\ell_6,\ell_1+\ell_2+2\ell_3+3\ell_4+\ell_5)\), 
\item \((0,\ell_1+\ell_2+2\ell_3+3\ell_4+\ell_5)\). 
\end{itemize}

The octagon \(\Delta_6(k,\ell_1,\ldots,\ell_6)\) corresponds to a Kähler class if and only if \(\ell_6-k\ell_5-(2k+1)\ell_4-(k+1)\ell_3-(2k+3)\ell_2-(k+2)\ell_1>0\), and then it has the following eight vertices
\begin{itemize}
\item \((0,0)\), 
\item \((\ell_6-k\ell_5-(2k+1)\ell_4-(k+1)\ell_3-(2k+3)\ell_2-(k+2)\ell_1,0)\), 
\item \((\ell_6-k\ell_5-(2k+1)\ell_4-(k+1)\ell_3-(2k+3)\ell_2,\ell_1)\), 
\item \((\ell_6-k\ell_5-(2k+1)\ell_4-(k+1)\ell_3,\ell_1+2\ell_2)\), 
\item \((\ell_6-k\ell_5-(2k+1)\ell_4,\ell_1+2\ell_2+\ell_3)\), 
\item \((\ell_6-k\ell_5,\ell_1+2\ell_2+\ell_3+2\ell_4)\), 
\item \((\ell_6,\ell_1+2\ell_2+\ell_3+2\ell_4+\ell_5)\), 
\item \((0,\ell_1+2\ell_2+\ell_3+2\ell_4+\ell_5)\). 
\end{itemize}

The octagon \(\Delta_7(k,\ell_1,\ldots,\ell_6)\) corresponds to a Kähler class if and only if \(\ell_6-k\ell_5-(4k+1)\ell_4-(3k+1)\ell_3-(2k+1)\ell_2-(k+1)\ell_1>0\), and then it has the following eight vertices
\begin{itemize}
\item \((0,0)\), 
\item \((\ell_6-k\ell_5-(4k+1)\ell_4-(3k+1)\ell_3-(2k+1)\ell_2-(k+1)\ell_1,0)\), 
\item \((\ell_6-k\ell_5-(4k+1)\ell_4-(3k+1)\ell_3-(2k+1)\ell_2,\ell_1)\), 
\item \((\ell_6-k\ell_5-(4k+1)\ell_4-(3k+1)\ell_3,\ell_1+2\ell_2)\), 
\item \((\ell_6-k\ell_5-(4k+1)\ell_4,\ell_1+2\ell_2+3\ell_3)\), 
\item \((\ell_6-k\ell_5,\ell_1+2\ell_2+3\ell_3+4\ell_4)\), 
\item \((\ell_6,\ell_1+2\ell_2+3\ell_3+4\ell_4+\ell_5)\), 
\item \((0,\ell_1+2\ell_2+3\ell_3+4\ell_4+\ell_5)\). 
\end{itemize}

The octagon \(\Delta_8(k,\ell_1,\ldots,\ell_6)\) corresponds to a Kähler class if and only if \(\ell_6-k\ell_5-(3k+1)\ell_4-(5k+2)\ell_3-(2k+1)\ell_2-(k+1)\ell_1>0\), and then it has the following eight vertices
\begin{itemize}
\item \((0,0)\), 
\item \((\ell_6-k\ell_5-(3k+1)\ell_4-(5k+2)\ell_3-(2k+1)\ell_2-(k+1)\ell_1,0)\), 
\item \((\ell_6-k\ell_5-(3k+1)\ell_4-(5k+2)\ell_3-(2k+1)\ell_2,\ell_1)\), 
\item \((\ell_6-k\ell_5-(3k+1)\ell_4-(5k+2)\ell_3,\ell_1+2\ell_2)\), 
\item \((\ell_6-k\ell_5-(3k+1)\ell_4,\ell_1+2\ell_2+5\ell_3)\), 
\item \((\ell_6-k\ell_5,\ell_1+2\ell_2+5\ell_3+3\ell_4)\), 
\item \((\ell_6,\ell_1+2\ell_2+5\ell_3+3\ell_4+\ell_5)\), 
\item \((0,\ell_1+2\ell_2+5\ell_3+3\ell_4+\ell_5)\). 
\end{itemize}

The octagon \(\Delta_9(k,\ell_1,\ldots,\ell_6)\) corresponds to a Kähler class if and only if \(\ell_6-k\ell_5-(3k+1)\ell_4-(2k+1)\ell_3-(3k+2)\ell_2-(k+1)\ell_1>0\), and then it has the following eight vertices
\begin{itemize}
\item \((0,0)\), 
\item \((\ell_6-k\ell_5-(3k+1)\ell_4-(2k+1)\ell_3-(3k+2)\ell_2-(k+1)\ell_1,0)\), 
\item \((\ell_6-k\ell_5-(3k+1)\ell_4-(2k+1)\ell_3-(3k+2)\ell_2,\ell_1)\), 
\item \((\ell_6-k\ell_5-(3k+1)\ell_4-(2k+1)\ell_3,\ell_1+3\ell_2)\), 
\item \((\ell_6-k\ell_5-(3k+1)\ell_4,\ell_1+3\ell_2+2\ell_3)\), 
\item \((\ell_6-k\ell_5,\ell_1+3\ell_2+2\ell_3+3\ell_4)\), 
\item \((\ell_6,\ell_1+3\ell_2+2\ell_3+3\ell_4+\ell_5)\), 
\item \((0,\ell_1+3\ell_2+2\ell_3+3\ell_4+\ell_5)\). 
\end{itemize}

The octagon \(\Delta_{10}(k,\ell_1,\ldots,\ell_6)\) corresponds to a Kähler class if and only if \(\ell_6+\ell_5-k\ell_4-(k+1)\ell_3-(k+2)\ell_2-(k+3)\ell_1>0\) and \(\ell_1+\ell_2+\ell_3+\ell_4-\ell_6>0\), and then it has the following eight vertices
\begin{itemize}
\item \((0,0)\), 
\item \((\ell_6+\ell_5-k\ell_4-(k+1)\ell_3-(k+2)\ell_2-(k+3)\ell_1,0)\), 
\item \((\ell_6+\ell_5-k\ell_4-(k+1)\ell_3-(k+2)\ell_2,\ell_1)\), 
\item \((\ell_6+\ell_5-k\ell_4-(k+1)\ell_3,\ell_1+\ell_2)\), 
\item \((\ell_6+\ell_5-k\ell_4,\ell_1+\ell_2+\ell_3)\), 
\item \((\ell_6+\ell_5,\ell_1+\ell_2+\ell_3+\ell_4)\), 
\item \((\ell_6,\ell_1+\ell_2+\ell_3+\ell_4)\), 
\item \((0,\ell_1+\ell_2+\ell_3+\ell_4-\ell_6)\). 
\end{itemize}

The octagon \(\Delta_{11}(k,\ell_1,\ldots,\ell_6)\) corresponds to a Kähler class if and only if \(\ell_6+\ell_5-k\ell_4-(2k+1)\ell_3-(k+1)\ell_2-(k+2)\ell_1>0\) and \(\ell_1+\ell_2+2\ell_3+\ell_4-\ell_6>0\), and then it has the following eight vertices
\begin{itemize}
\item \((0,0)\), 
\item \((\ell_6+\ell_5-k\ell_4-(2k+1)\ell_3-(k+1)\ell_2-(k+2)\ell_1,0)\), 
\item \((\ell_6+\ell_5-k\ell_4-(2k+1)\ell_3-(k+1)\ell_2,\ell_1)\), 
\item \((\ell_6+\ell_5-k\ell_4-(2k+1)\ell_3,\ell_1+\ell_2)\), 
\item \((\ell_6+\ell_5-k\ell_4,\ell_1+\ell_2+2\ell_3)\), 
\item \((\ell_6+\ell_5,\ell_1+\ell_2+2\ell_3+\ell_4)\), 
\item \((\ell_6,\ell_1+\ell_2+2\ell_3+\ell_4)\), 
\item \((0,\ell_1+\ell_2+2\ell_3+\ell_4-\ell_6)\). 
\end{itemize}

The octagon \(\Delta_{12}(k,\ell_1,\ldots,\ell_6)\) corresponds to a Kähler class if and only if \(\ell_6+\ell_5-k\ell_4-(3k+1)\ell_3-(2k+1)\ell_2-(k+1)\ell_1>0\) and \(\ell_1+2\ell_2+3\ell_3+\ell_4-\ell_6>0\), and then it has the following eight vertices
\begin{itemize}
\item \((0,0)\), 
\item \((\ell_6+\ell_5-k\ell_4-(3k+1)\ell_3-(2k+1)\ell_2-(k+1)\ell_1,0)\), 
\item \((\ell_6+\ell_5-k\ell_4-(3k+1)\ell_3-(2k+1)\ell_2,\ell_1)\), 
\item \((\ell_6+\ell_5-k\ell_4-(3k+1)\ell_3,\ell_1+2\ell_2)\), 
\item \((\ell_6+\ell_5-k\ell_4,\ell_1+2\ell_2+3\ell_3)\), 
\item \((\ell_6+\ell_5,\ell_1+2\ell_2+3\ell_3+\ell_4)\), 
\item \((\ell_6,\ell_1+2\ell_2+3\ell_3+\ell_4)\), 
\item \((0,\ell_1+2\ell_2+3\ell_3+\ell_4-\ell_6)\). 
\end{itemize}

The octagon \(\Delta_{13}(k,\ell_1,\ldots,\ell_6)\) corresponds to a Kähler class if and only if \(\ell_6+2\ell_5+\ell_4-k\ell_3-(k+1)\ell_2-(k+2)\ell_1>0\) and \(\ell_1+\ell_2+\ell_3-\ell_5-\ell_6>0\), and then it has the following eight vertices
\begin{itemize}
\item \((0,0)\), 
\item \((\ell_6+2\ell_5+\ell_4-k\ell_3-(k+1)\ell_2-(k+2)\ell_1,0)\), 
\item \((\ell_6+2\ell_5+\ell_4-k\ell_3-(k+1)\ell_2,\ell_1)\), 
\item \((\ell_6+2\ell_5+\ell_4-k\ell_3,\ell_1+\ell_2)\), 
\item \((\ell_6+2\ell_5+\ell_4,\ell_1+\ell_2+\ell_3)\), 
\item \((\ell_6+2\ell_5,\ell_1+\ell_2+\ell_3)\), 
\item \((\ell_6,\ell_1+\ell_2+\ell_3-\ell_5)\), 
\item \((0,\ell_1+\ell_2+\ell_3-\ell_5-\ell_6)\). 
\end{itemize}

The octagon \(\Delta_{14}(k,\ell_1,\ldots,\ell_6)\) corresponds to a Kähler class if and only if \(\ell_6+2\ell_5+\ell_4-k\ell_3-(2k+1)\ell_2-(k+1)\ell_1>0\) and \(\ell_1+2\ell_2+\ell_3-\ell_5-\ell_6>0\), and then it has the following eight vertices
\begin{itemize}
\item \((0,0)\), 
\item \((\ell_6+2\ell_5+\ell_4-k\ell_3-(2k+1)\ell_2-(k+1)\ell_1,0)\), 
\item \((\ell_6+2\ell_5+\ell_4-k\ell_3-(2k+1)\ell_2,\ell_1)\), 
\item \((\ell_6+2\ell_5+\ell_4-k\ell_3,\ell_1+2\ell_2)\), 
\item \((\ell_6+2\ell_5+\ell_4,\ell_1+2\ell_2+\ell_3)\), 
\item \((\ell_6+2\ell_5,\ell_1+2\ell_2+\ell_3)\), 
\item \((\ell_6,\ell_1+2\ell_2+\ell_3-\ell_5)\), 
\item \((0,\ell_1+2\ell_2+\ell_3-\ell_5-\ell_6)\). 
\end{itemize}

The octagon \(\Delta_{15}(k,\ell_1,\ldots,\ell_6)\) corresponds to a Kähler class if \((k+1)\ell_1+(2k+1)\ell_2+k\ell_3-\ell_4-\ell_5-\ell_6<0\) and \(\ell_1+2\ell_2+\ell_3-\ell_5-2\ell_6>0\), and then it has the following eight vertices
\begin{itemize}
\item \((0,0)\), 
\item \((\ell_6-(k+1)\ell_1-(2k+1)\ell_2-k\ell_3+\ell_4+\ell_5,0)\), 
\item \((\ell_6-(2k+1)\ell_2-k\ell_3+\ell_4+\ell_5,\ell_1)\), 
\item \((\ell_6-k\ell_3+\ell_4+\ell_5,\ell_1+2\ell_2)\), 
\item \((\ell_6+\ell_4+\ell_5,\ell_1+2\ell_2+\ell_3)\), 
\item \((\ell_6+\ell_5,\ell_1+2\ell_2+\ell_3)\), 
\item \((\ell_6,\ell_1+2\ell_2+\ell_3-\ell_5)\), 
\item \((0,\ell_1+2\ell_2+\ell_3-\ell_5-2\ell_6)\). 
\end{itemize}
\end{proof}

\begin{rem}
In general, one should parametrize the families by \(k\in \bbZ\) but there are obvious symmetries in certain families (for example, it is enough to take \(k \geq -2\) for \(\Sigma_1\)), as well as some sporadic identifications (for example between \(\Sigma_1(-2)\) and \(\Sigma_{15}(-2)\)). 
\end{rem}

\begin{exa}
Here is an example of Delzant octagon from family 15. 
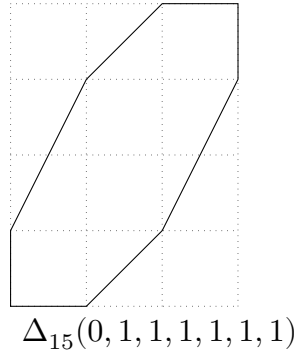
\begin{center}
\begin{tikzpicture}
\draw[dotted,very thin] (0,0) grid (3,4);
\draw (0,0) node[below right]{\(\Delta_{15}(0,1,1,1,1,1,1)\)} -- (1,0) -- (2,1) -- (3,3) -- (3,4) -- (2,4) -- (1,3) -- (0,1) -- cycle;
\end{tikzpicture}
\end{center}
\end{exa}

\section{The tentative examples of Wang and Zhou: a dead end?}

In \cite{Wang-Zhou_2014}, the authors consider the following family of Delzant octagons: its vertices are given by 
\((-1,-\alpha-2k), (-1,\alpha+4), (0,\alpha+4), (\epsilon_1,\alpha+4-\epsilon_1), (\epsilon_1+\epsilon_2,\alpha+4-\epsilon_1-2\epsilon_2), (\epsilon_1+\epsilon_2+\epsilon_3,\alpha+4-\epsilon_1-2\epsilon_2-3\epsilon_3), (1,\alpha+3\epsilon_1+2\epsilon_2+\epsilon_3), (1,-\alpha)\), 
for some positive parameters \(k\in \bbZ_{>0}\), \(\alpha\), \(\epsilon_1\), \(\epsilon_2\) and \(\epsilon_3\). 
Under the action of \(\begin{pmatrix}
0 & -1 \\ 1 & k
\end{pmatrix}\), it is easy to check that these polytopes are of the form \(\Delta_1(k,\ell_1,\ldots,\ell_6)\) where the notation \(k\) is consistent, 
\(\ell_1 = 1-\epsilon_1-\epsilon_2-\epsilon_3\),
\(\ell_2 = \epsilon_3\),
\(\ell_3 = \epsilon_2\),
\(\ell_4 = \epsilon_1\),
\(\ell_5 = 1\), and 
\(\ell_6 = 2\alpha+2k+4\).
Our numerical exploration did not uncover relatively K-unstable examples of this type as will be explained in the next sections. 
Let us explain this discrepancy by reviewing and correcting Wang and Zhou's argument. 

The claim in \cite{Wang-Zhou_2014} is that, writing \(\alpha=\beta k^2\), the octagon described above is relatively K-unstable for \(\beta\) fixed \(\geq 1/3\), \(k\) large enough, and \(\epsilon_1\), \(\epsilon_2\), \(\epsilon_3\) small enough. 
More precisely, the claim follows from the other claim that the limiting pentagon as all \(\epsilon_i\to 0\) is relatively K-unstable as well. 
This is the pentagon \(P\) whose vertices are 
\((-1,-\alpha-2k), (-1,\alpha+4), (0,\alpha+4), (1,\alpha), (1,-\alpha)\). 
More precisely again, the claim is that this pentagon is destabilized by the simple piecewise linear function 
\[u_t=\max(y-\alpha+kt,0) \]
for \(0<< t << k\). 
We will disprove this claim. 

\begin{thm}
The pentagon \(P\) is not destabilized by any function of the form \(u_t\) for \(k\) large enough and \(0< t << k\). 
\end{thm}

\begin{proof}
The quantities involved in the Donaldson-Futaki invariant for this polygon (as for any polygon) can be computed with elementary methods: for an arbitrary function, 
\begin{align*}
\int_{\partial P} f d\sigma = & 
\int_0^{2\alpha+2k+4} f(-1,-\alpha-2k+t)dt \\
&+ \int_0^1 f(-1+t,\alpha+4) dt \\
&+ \int_0^1 f(t, \alpha+4-4t)dt \\
&+ \int_0^{2\alpha} f(1,\alpha-t) dt \\
&+ \int_0^2 f(1-t,-\alpha-kt) dt
\end{align*}
and 
\begin{align*}
\int_P f dx = & 
\int_{x=-1}^0\int_{y=-\alpha-k+kx}^{\alpha+4} f(x,y) dy dx \\
& + \int_{x=0}^1\int_{y=-\alpha-k+kx}^{\alpha+4-4x} f(x,y) dy dx 
\end{align*}

Using these expression, one readily determines the explicit linear system \(M(a_0,a_1,a_2)^T=(b_0,b_1,b_2)^T\), as in \cite{Wang-Zhou_2014}, such that the extremal affine function is \(A(x,y)=a_0+a_1x+a_2y\) in these coordinates. 
The authors of that paper derive asymptotic information on the coefficients \(a_0\), \(a_1\) and \(a_2\) and notably that 
\[ a_0 = 1+ \frac{3\beta -1}{6\beta^2k^2} + O(k^{-3})\]
\[ a_1 = \frac{-1}{\beta k} + O(k^{-2}) \]
and 
\[ a_2 = O(k^{-4})\]
(recall that \(\alpha=\beta k^2\)). 
This is correct as can be checked by using a variant of our SageMath program to compute \(A\) explicitly, but this is also not precise enough: one may more precisely write 
\[ a_2 = \frac{-3}{2\beta^2k^4} + O(k^{-5})\]
indeed, exact computation gives \(a_2=-\frac{3}{2}\frac{p}{q}\) where 
\[ p = 24\beta^3k^6 + 36\beta^2k^5 + 94\beta^2k^4 + 14\beta k^4 + 78\beta k^3 + 108\beta k^2 + k^3 + 10k^2 + 32k + 33 \]
and 
\begin{align*}
q= & 24\beta^5k^{10} + 60\beta^4k^9 + 174\beta^4k^8 + 64\beta^3k^8 + 364\beta^3k^7 + 532\beta^3k^6 + 36\beta^2k^7 \\  
&+ 300\beta^2k^6 + 858\beta^2k^5 + 10\beta k^6 + 840\beta^2k^4 + 110\beta k^5 + 464\beta k^4 + k^5 \\ &+ 892\beta k^3 + 14k^4 + 658\beta k^2 + 79k^3 + 226k^2 + 329k + 195
\end{align*}

Assuming \(0<t<<k\) as in \cite{Wang-Zhou_2014}, we have \(0<<\alpha-tk<\alpha\), which is the ordinate at which the simple piecewise linear function \(u_t\) starts to be non-zero.
As a consequence, the crease of the simple piecewise linear function \(u_t\) connects the two vertical sides of the pentagon. 
Again using formal computations, we have 
\begin{align*}
\mathcal{L}(u_t) = &k^2t^2 + 6 k t +14 \\
&-a_0 (k^2t^2+6kt + \frac{32}{3}) \\
&-a_1 (-\frac{4}{3}kt-\frac{10}{3}) \\
&-a_2 (\alpha k^2 t^2 - \frac{1}{3}k^3t^3 + 6 \alpha kt + \frac{32}{3}\alpha + \frac{32}{3}kt + \frac{80}{3})
\end{align*} 
It is transparent from the latter expression that the asymptotic behavior of \(a_2\) does play a role in the asymptotic behavior of \(\mathcal{L}(u_t)\). 
By using the asymptotic information on the \(a_i\), and taking the limit as \(k\to \infty\), we get 
\[ \lim_{k\to \infty}\mathcal{L}(u_t)= \frac{10}{3} - \frac{4t}{3\beta} + \frac{(6\beta+1)t^2}{6\beta^2} \]
Since \(\beta>0\), the latter function, quadratic in \(t\), is strictly positive by checking directly that its minimal value is \(\frac{2(30\beta +1)}{6\beta+1}\).
\end{proof}

\begin{rem}
Note that this is not a proof that \(P\) is relatively K-stable, even though our numerical exploration indicates that it probably is. 
It is in general very intricate to prove relative K-stability, even for pentagons, see \cite{LuoChen2020}.  
\end{rem}

\section{Present numerical and symbolic implementations}
This section describes our numerical methodology with a posteriori symbolic verification over the rational field.

\subsection{Delzant octagon families}
We recall that each family is defined explicitly by a set of eight vertices
$P=\operatorname{Conv}(v_0,\ldots,v_7).$
The coordinates are affine functions of $(l_1,l_2,l_3,l_4,l_5,l_6,k).$
The numerical implementation contains fifteen explicit constructions corresponding to distinct Delzant configurations. 

For every candidate parameter set the polygon is constructed, 
convexity is verified, Delzant primitive edge directions are assigned,
geometric quantities are computed. Only convex polygons are retained.

\subsection{Delzant boundary measure}
Let $e=[A,B]$ denote an edge of the polygon and let
$v=(v_{x_1},v_{x_2})$
be the primitive integral direction associated with that edge.
If
$(d{x_1},d{x_2})=B-A,$
the Delzant boundary weight is defined as
$\sigma
=
\left|
\frac{d{x_1}}{v_{x_1}}
\right|
$
whenever \(v_{x_1}\neq0\), and
$
\sigma
=
\left|
\frac{d{x_2}}{v_{x_2}}
\right|
$
otherwise.

Boundary integrals therefore take the form

\[
\int_{\partial P}f\,d\sigma
=
\sum_e
\sigma_e
\int_0^{l_e}f_e(t)\,dt.
\]

\subsection{The extremal affine function}
The extremal affine function is written
$A({x_1},{x_2})=\alpha_0+\alpha_1 {x_1}+\alpha_2 {x_2}.$
Its coefficients are determined by requiring that the Donaldson-Futaki invariant vanishes on all affine functions. Equivalently, \(A\) may be characterized as the affine function whose moments reproduce the boundary moments of the polygon. Introduce the affine basis
\[
\phi_0=1,
\qquad
\phi_1={x_1},
\qquad
\phi_2={x_2}.
\]
Let $\mathcal A = \operatorname{span}\{1,{x_1},{x_2}\}$
denote the three-dimensional space of affine functions on \(P\).
The extremal affine function can be viewed as the orthogonal projection of the boundary measure onto \(\mathcal A\) with respect to the interior \(L^2\)-pairing. More precisely, one seeks an affine function
\[
A=\sum_{j=0}^{2}\alpha_j\phi_j
\]
such that
\[
\int_P A\,\phi_i\,d\mu = \int_{\partial P}\phi_i\,d\sigma, \qquad i=0,1,2.
\]
This is the normal equation associated with the least-squares problem
\[
\min_{A\in\mathcal A} \sum_{i=0}^{2} \left( \int_P A\,\phi_i\,d\mu - \int_{\partial P}\phi_i\,d\sigma \right)^2 .
\]
Define the moment matrix
\[ M_{ij} = \int_P \phi_i\phi_j\,d\mu \]
and the boundary moment vector
\[ b_i =\int_{\partial P} \phi_i\,d\sigma. \]
Substituting $A$ into the moment conditions yields the linear system
\[ \sum_{j=0}^{2} M_{ij}\alpha_j = b_i, \qquad i=0,1,2,\]
or, in matrix form, $M\alpha=b.$
Since \(M\) is a Gram matrix associated with the interior \(L^2\)-pairing, it is symmetric and positive definite whenever the polygon has nonzero area. The coefficient vector is therefore uniquely determined.

\subsection{Piecewise linear test configurations}
Given two points
$p_1 \text{ and } p_2 $ chosen on two distinct polygon edges, the line
$L({x_1},{x_1})=a{x_1}+b{x_2}+c $
passing through these points is constructed.
The associated convex piecewise-linear function is
$u({x_1},{x_2}) = \max(0,L({x_1},{x_2})).$
Only non-degenerate cuts are retained, namely cuts intersecting the polygon boundary transversely in exactly two distinct points.

The positivity region is
$ P_+ = \{
({x_1},{x_2})\in P
\;:\;
L({x_1},{x_2})>0
\}.
$

\subsection{Donaldson-Futaki invariant}
For every admissible test function \(u\), the Donaldson-Futaki invariant is
$L_A(u)=\int_{\partial P}u\,d\sigma-\int_P A\,u\,d\mu.$

\subsection{Numerical integration}
The software evaluates the functional using adaptive quadrature.
Boundary contributions are computed using one-dimensional integration over each edge.
Area contributions are computed by repeated evaluation of
$ \int_P A\,u\,d\mu $
using two-dimensional adaptive quadrature.
The numerical tolerances are chosen at the level $10^{-14}.$
A fixed random seed is used to ensure reproducibility of the Monte-Carlo exploration.

\subsection{Optimization/minimization strategy}
Minimization search takes place in the admissible parameter space where for a given parameter vector $x=(l_1,l_2,l_3,l_4,l_5)$ (we fix \(l_6=1\), new candidates are generated by multiplicative perturbations
$ x_{\text{new}}   = x_{\text{old}} \exp(\eta), $
where \(\eta\) is Gaussian noise.
The algorithm accepts admissible candidates whenever
$ L_A(x_{\text{new}}) < L_A(x_{\text{best}}).$
This generates a greedy descent toward increasingly unstable configurations.
The smallest value encountered together with the associated geometry is stored.

\subsection{Examples of numerical results provided by the software}
Any value of the integer parameter \(k\) can be tested as well as a sweep over a range of \(k\).
The numerical results shown here are for exploration performed for all fifteen Delzant 
octagon families at \(k=20\). 
For each family, the minimization procedure identified the most unstable admissible configuration. 
An independent symbolic reconstruction was then carried out to (try to) confirm instability
remains under exact rational reconstruction.

\subsubsection{Consistency of the numerical pipeline}
The software also proceeds with self a posteriori validation: for every family tested, the optimal value is recomputed independently to make sure the pipeline is consistent. Typical discrepancies between the stored value and the recomputed value are between $10^{-30} \text{and } 10^{-18},$
depending on the magnitude of the functional.
This demonstrates that the numerical pipeline is internally stable and that the stored configurations can be reconstructed reliably from the recorded data and stored in {\tt json} files.

\subsubsection{Numerical ninima}
Table~\ref{tab:numeric_vs_symbolic} summarizes the best values of the Donaldson-Futaki invariant for \(k=20\).
Cases 10-15 behave differently than the first nine families.

\subsubsection{Sensitivity to floating-point precision}
The implementation also allows the user to prescribe the floating-point precision employed during the optimization stage. This makes it possible to investigate the dependence of the resulting destabilizing configurations on the arithmetic precision used by the numerical search.

For Case~15, for instance, the optimization was repeated using two different output precisions: 20 decimal digits and 30 decimal digits. The resulting configurations were then compared.
The value of Donaldson's functional remained essentially unchanged:

$
L_A^{(20)}
=
-1.4127889871597290039,
$
and
$
L_A^{(30)}
=
-1.41278898715972900390625.
$

Similarly, the affine coefficients showed agreement to many significant digits:
\begin{table}[ht]
\centering
\caption{Comparison of the affine coefficients obtained with 20 and 30 decimal digits.}
\begin{tabular}{lcc}
\hline
Coefficient & 20 digits & 30 digits \\
\hline
$\alpha_0$
&
$0.15146206420847818652$
&
$0.151462064208478186522199848696$
\\[1ex]

$\alpha_1$
&
$2.6433849878925404605\times 10^{-11}$
&
$2.64338498789254046050853085369\times 10^{-11}$
\\[1ex]

$\alpha_2$
&
$9.7393624302492321811\times 10^{-4}$
&
$9.73936243024923218106203570699\times 10^{-4}$
\\
\hline
\end{tabular}
\end{table}

The geometric data likewise remained stable. For example, the destabilizing points defining the test configuration differed only beyond the displayed precision.

These observations indicate that the numerical optimization is stable with respect to the chosen floating-point precision. In particular, the sign of Donaldson's functional remains unchanged, and the resulting destabilizing configuration exhibits no noticeable geometric variation when the precision is increased from 20 to 30 decimal digits.

\begin{rem}
There is a possibility that the algorithms used to compute \(A\) introduces significant numerical errors in close to degenerate cases, this should be investigated in further uses of the program. 
This is also addressed in the next section. 
\end{rem}

\subsection{Rational reconstruction}
At this point of the calculation, a negative value of $L_A$ is only indication of possible instability.
To make sure candidates are true instabilities, We proceed with an additional independent verification with {\tt SageMath} to determine whether the instabilities survive the replacement of all floating-point data by exact rational numbers using its decimal representation. 
In particular, the polygon parameters $l_1,\ldots,l_5,$ the line-defining points $p_1,\;p_2,$ and the affine coefficients $\alpha=(\alpha_0,\alpha_1,\alpha_2)$ were reconstructed as solution of  $M\alpha=b$ in \(\mathbf Q\).
The resulting Delzant octagon therefore lies entirely in \(\mathbf Q^2\), and the affine function
$
A({x_1},{x_2})=\alpha_0+\alpha_1 {x_1}+\alpha_2 {x_2}
$
possesses exact rational coefficients.

The destabilizing line was reconstructed from the stored points $p_1, p_2,$ using exact rational. The verification yielded $L(p_1)=0, L(p_2)=0$.
Consequently, the reconstructed line passes through the stored points without any numerical approximation.
The positivity region $ P_+ = \{({x_1},{x_2})\in P:\;L({x_1},{x_2})>0\} $
was then reconstructed using exact sign tests at the polygon vertices together with exact computations of the boundary intersection points. All vertices of the resulting positivity polygon were represented by rational coordinates.
Using the affine function reconstructed over \(\mathbf Q\), the boundary and area contributions were evaluated symbolically and also shown in Table~\ref{tab:numeric_vs_symbolic}.

\begin{table}[ht]
\centering
\begin{tabular}{cccc}
\toprule
Case &
Numerical \(L_A\) &
Symbolic \(L_A\) \\
\midrule
1  & $-1.27...\times10^{-11}$ & $1.63...\times10^{-11}$ \\
2  & $-6.13...\times10^{-13}$ & $1.83...\times10^{-13}$ \\
3  & $-2.18...\times10^{-11}$ & $2.22...\times10^{-12}$ \\
4  & $-1.39...\times10^{-12}$ & $1.38...\times10^{-13}$ \\
5  & $-2.78...\times10^{-12}$ & $3.31...\times10^{-13}$ \\
6  & $-9.51...\times10^{-12}$ & $3.18...\times10^{-13}$ \\
7  & $-1.51...\times10^{-10}$ & $2.36...\times10^{-10}$ \\
8  & $-2.11...\times10^{-11}$ & $6.78...\times10^{-12}$ \\
9  & $-3.86...\times10^{-12}$ & $1.00...\times10^{-12}$ \\
10 & $-0.071...$ & $0.01505...$ \\
11 & $-0.20...$ & $0.01920...$ \\
12 & $-0.19...$ &  $0.05482...$ \\
13 & $-320.28...$ & $1.13443...$ \\
14 & $-12.77...$ & $0.41400...$ \\
15 & $-1.41...$ & $0.00448...$ \\
\bottomrule
\end{tabular}
\caption{Comparison between numerical and symbolic minimum $L_A$.}
\label{tab:numeric_vs_symbolic}
\end{table}

The destabilizing configurations found numerically in real field do not resist under exact rational reconstruction and symbolic verification. 

\section{Exact Rational Monte Carlo Search}
To eliminate floating-point effects entirely, a second search was carried out directly over the field \(\mathbf Q\). Rational values of the parameters $(l_1,l_2,l_3,l_4,l_5)$
were generated by a Monte Carlo procedure subject to the admissibility constraints of the corresponding Delzant octagon. Two rational points were then chosen on distinct edges of the polygon, defining a rational piecewise-linear test configuration.
For each sample, every quantity was computed exactly in \(\mathbf Q\):
\begin{enumerate}
\item construction of the Delzant polygon,
\item computation of the moment matrix \(M\),
\item computation of the boundary moment vector \(b\),
\item solution of the extremal affine equation
\[
M\alpha=b,
\]
\item construction of the positive region \(P_+\),
\item exact evaluation of the area term,
\item exact evaluation of the boundary term,
\item exact computation of the Donaldson functional \(L_A\).
\end{enumerate}
Thus no numerical quadrature, floating-point arithmetic, or rational reconstruction was involved.

Among the fifteen families considered, Case 15 appeared to be the most likely location for  instability. However, an exact search over rational parameter values failed to produce a negative value of Donaldson's functional. The smallest value found was approximately $4.92...\times10^{-5}$, below the value  $4.48...\times10^{-3}$ obtained by the floating-point search, yet remaining strictly positive.

\section{Physical interpretation, connection with Green's formula and Poisson theory}
The Donaldson functional resembles the structure of Green's identities. Indeed, if \(u\) and \(v\) are sufficiently regular functions on a planar domain \(P\), Green's formula gives
\[
\int_P \nabla u\cdot\nabla v\,dx
=
-\int_P u\,\Delta v\,dx
+
\int_{\partial P}
u\,\frac{\partial v}{\partial n}\,d\sigma,
\]
where \(\partial v/\partial n\) denotes the outward normal derivative.
The Donaldson functional compares a boundary contribution with a bulk contribution, 
as in Green's identity. The affine function $A(x)$
plays the role of an interior source term. The coefficients
\((\alpha_0,\alpha_1,\alpha_2)\) are chosen so that the functional vanishes on all affine functions,
\[
L_A(1)=L_A(x_1)=L_A(x_2)=0.
\]
Let
\[
\phi=(1,x_1,x_2).
\]
The normalization may be written as
\[
\int_P A(x)\,\phi_i(x)\,dx
=
\int_{\partial P}\phi_i\,d\sigma,
\qquad i=0,1,2.
\]
Expanding
\[
A(x)=\sum_{j=0}^{2}\alpha_j\phi_j(x),
\]
one obtains the linear system $M\alpha=b,$
where
$
M_{ij}
=
\int_P \phi_i\phi_j\,dx,
$
and
$
b_i
=
\int_{\partial P}\phi_i\,d\sigma.
$

$A$ is therefore a moment-matching density. It provides the unique affine interior density whose zeroth moment (or total mass: $ \int_{\Delta} A\, d\mu = \int_{\partial\Delta} 1\, d\sigma,$ and first moment
(or barycentric measure: $\int_{\Delta} xA\, d\mu = \int_{\partial\Delta} x\, d\sigma$)
moments agree with those of the boundary measure. 
In other words, $A$ is the unique affine distribution reproducing the monopole and dipole structure of the boundary measure. Indeed, the construction admits a natural interpretation in terms of the Poisson equation
$-\Delta\Phi=\rho,$
where \(\rho\) represents a bulk source and \(\Phi\) the associated potential.
If \(\Phi\) satisfies
$-\Delta\Phi=A$
inside the polygon and
$
\frac{\partial\Phi}{\partial n}=1
$
on the boundary, Green's identity yields
$
L_A(u)
=
-\int_P \nabla\Phi\cdot\nabla u\,dx.
$

From this viewpoint, the affine function \(A\) plays the role of a bulk charge density;
 the boundary measure \(d\sigma\) acts as a uniform boundary source;
  \(\Phi\) is the potential generated by these sources;
   \(u\) is a test potential.

The normalization conditions
$L_A(1)=L_A(x_1)=L_A(x_2)=0$
state that the bulk source \(A\) reproduces the total mass and first moments of the boundary distribution. In physical language, the monopole and dipole moments of the interior and boundary contributions coincide.

If a destabilizing configuration is found in the toric setting 
for the test function $u(x)=\max\{0,L(x)\},$
where \(L(x)=0\) is a line intersecting the polygon, it can be interpreted
physically as a piecewise-linear potential or a deformation mode 
of a membrane occupying the polygon \(P\).
$ L_A(u)$ measures the balance between boundary effects and bulk effects for this particular mode. If $L_A(u)<0$ the bulk contribution weighted by \(A\) dominates the boundary contribution. 

Another interesting and intuitive physical interpretation is that the absence of negative values of $L_A(u)$ among the tested functions indicates that no energy-decreasing mode was detected within the sampled family of piecewise-linear test configurations.
Actually, the configurations appear stable under the discretized perturbations represented by the rational search, while continuous search reveals genuine energy-lowering modes.

\section{Conclusions}

We have developed a combined numerical and symbolic framework for studying relative K-stability of families of Delzant octagons. The implementation performs large-scale searches for destabilizing piecewise-linear test configurations, reconstructs the resulting geometry symbolically, and verifies the numerical computations independently.
The framework identifies explicit "candidate" configurations for which Donaldson's functional becomes negative and provides a symbolic reconstruction of the corresponding floating-point data together with recalculation of the extremal function $A$ in rational field.
None of the "candidate" unstable configurations found by numerical minimization has survived this rational reconstruction and symbolic verification. 
As a consequence the family of octagons exhibited in Wang and Zhou's article may very well actually be relatively K-stable, and their initial argument is corrected here. 

Independent minimization carried out entirely over the rational field likewise failed to produce any destabilizing configuration.
Hence, minimization over real values does not appear suitable for the discovery of unstable rational configurations in this precise setting. 

The computational framework developed for this article may be adapted with limited efforts to different settings were it should instead allow the user to find new relatively K-unstable examples, or to get numerical evidence of relative K-stability. 
The authors used a preliminary version of the code to numerically verify Luo and Chen's result in \cite{LuoChen2020}. 
For example, the program could be used as a tool to establish a numerical cartography of relatively K-unstable nonagons, whose existence was established by Wang and Zhou in \cite{Wang-Zhou_2011} and for which an explicit example is provided in Sektnan's PhD thesis \cite{Sektnan_2016_PhD}.  
If one is willing to consider non smooth toric surfaces, one may quite easily find relatively K-unstable quadrilaterals among orbifold toric surfaces \cite{Legendre_2011,Apostolov_Calderbank_Gauduchon_2015_ambitoric2}, but there is to our knowledge no completely explicit example given of relatively K-unstable normal toric surface in the literature. 
Numerical tools may also help in the cartography of relatively K-unstable normal toric surfaces. 

\section{Use of Artificial Intelligence Tools}

Microsoft Copilot was used as an auxiliary redaction and programming tool during the implementation phase of this work. Software development tasks concerned the generation of {\tt Python} and {\tt SageMath} codes from formulas in {\tt Latex}, construction of vertex data structures for Delzant polygons, and automated verification routines for strict convexity and edge-direction consistency. OpenAI's ChatGPT 5.6 and Microsoft Copilot were used to investigate independently the existence of relative K-unstable octagons, with no success. The mathematical arguments and interpretation of the computational results were performed and validated by the authors. The codes were reviewed, corrected and tested by the authors before being used in the computations reported in this article. 
 
\section{Software Availability and Reproducibility}
 
The computational framework developed in this work is publicly available at:

\begin{center}
\url{https://github.com/bijanmohammadi/K_Stability_Donaldson_Delzant.git}
\end{center}

The repository contains:

\begin{itemize}
\item the numerical search script for Delzant octagon families in {\tt Python};
\item the symbolic verification script in {\tt  SageMath} reading {\tt json} files generated by the 
 {\tt Python} code;
\item {\tt json} files of the 15 solutions presented in the paper which can be checked with the {\tt  SageMath} script;
\item the symbolic search Monte Carlo script for Delzant octagons in {\tt SageMath}.
\end{itemize}

\bibliographystyle{alpha}
\bibliography{NCRKO}

\end{document}